\documentclass[numbers,webpdf,imanum]{ima-authoring-template}%

\graphicspath{{Fig/}}
\usepackage{bm, graphicx, graphics}
\usepackage{multirow, threeparttable}

\theoremstyle{thmstyletwo}%
\newtheorem{theorem}{Theorem}
\newtheorem{remark}{Remark}%

\numberwithin{equation}{section}
\newtheorem{lemma}[theorem]{Lemma}

\newcommand{\ds}{\displaystyle}

\begin{document}

\DOI{DOI HERE}
\copyrightyear{2021}
\vol{00}
\pubyear{2021}
\access{Advance Access Publication Date: Day Month Year}
\appnotes{Paper}
\copyrightstatement{Published by Oxford University Press on behalf of the Institute of Mathematics and its Applications. All rights reserved.}
\firstpage{1}


\title[Semi-implicit method of high-index saddle dynamics]{Semi-implicit method of high-index saddle dynamics and application to construct solution landscape}

\author{Yue Luo
\address{\orgdiv{Beijing International Center for Mathematical Research},
	\orgname{Peking University}, 
	\orgaddress{\state{Beijing},   
		        \postcode{100871},
		         \country{China}}}
	         }

\author{Lei Zhang*
	\address{\orgdiv{Beijing International Center for Mathematical Research, Center for Machine Learning Research, Center for Quantitative Biology},
		\orgname{Peking University}, 
		\orgaddress{
			\state{Beijing}, 	
			\postcode{100871}, 
			\country{China}}}}

\author{Pingwen Zhang
	\address{\orgdiv{School of Mathematics and Statistics},
		\orgname{Wuhan University}, 
		\orgaddress{
			\state{Wuhan}, 
			\postcode{430072}, 
			\country{China}}
		}
	\address{\orgdiv{School of Mathematical Sciences, Laboratory of Mathematics and
			Applied Mathematics},
		\orgname{Peking University}, 
		\orgaddress{
			\state{Beijing},\postcode{100871},  
			\country{China}}}}

\author{Zhiyi Zhang
	\address{\orgdiv{School of Mathematical Sciences},
		\orgname{Peking University}, 
		\orgaddress{	
			\state{Beijing}, 
			\postcode{100871}, 
			\country{China}}}}
							         
\author{Xiangcheng Zheng
	\address{\orgdiv{School of Mathematics},
		\orgname{Shandong University}, 
		\orgaddress{
			\state{Jinan}, 
			\postcode{250100},
			\country{China}}}}

\authormark{Yue Luo et al.}

\corresp[*]{Corresponding author: \href{email:zhangl@math.pku.edu.cn}{zhangl@math.pku.edu.cn}}

\received{Date}{0}{Year}
\revised{Date}{0}{Year}
\accepted{Date}{0}{Year}


\abstract{	We analyze the semi-implicit scheme of high-index saddle dynamics, which provides a powerful numerical method for finding the any-index saddle points and constructing the solution landscape. Compared with the explicit schemes of saddle dynamics, the semi-implicit discretization relaxes the step size and accelerates the convergence, but the corresponding numerical analysis encounters new difficulties compared to the explicit scheme. Specifically, the orthonormal property of the eigenvectors at each time step could not be fully employed due to the semi-implicit treatment, and computations of the eigenvectors are coupled with the orthonormalization procedure, which further complicates the numerical analysis. We address these issues to prove error estimates of the semi-implicit scheme via, e.g. technical splittings and multi-variable circulating induction procedure. We further analyze the convergence rate of the generalized minimum residual solver for solving the semi-implicit system. Extensive numerical experiments are carried out to substantiate the efficiency and accuracy of the semi-implicit scheme in constructing solution landscapes of complex systems.}

\maketitle

\section{Introduction}
Searching saddle points on a complicated energy landscape is a crucial but challenging topic in computational physical and chemistry \cite{EV2010,Han2019transition, HanXu,Li2001, Mehta, Shi2022,Wal, wang2021modeling, Yin2020nucleation,YuZhe, npj2016,ZhangChe}. The saddle points can be classified by the (Morse) index, which is characterized by the maximal dimension of a subspace on which the Hessian is negative definite, according to the Morse theory \cite{Milnor}.
Most existing numerical algorithms focus on finding the index-1 saddle points, e.g. \cite{baker1986,Doye,EZho,Farr,Gao,Gou,Dimer,Lev,
	ZhaDu}. However, the computation of high-index saddle points receive less attention due to their unstable nature, despite of the fact that the number of high-index saddle points are much larger than local minima and index-1 saddle points on the complicated energy landscapes \cite{YinPRL,YinSCM} {\color{blue} and their important applications among many scientific fields. For instance, in chemical system, index-2 saddle points offer valuable information on the trajectories of chemical reactions \cite{Quapp1986, Quapp2015, Wiggins2013, Jasper2012, Quapp2004}. In quantum mechanics, excited states that have higher energy than the ground state are modeled as high-index saddle points of the energy functional with the constraint \cite{YinBE, Bao2012, bao2004}. } 

{\color{blue}High-index saddle dynamics (HiSD) is recently proposed in \cite{YinSISC} to find the any-index saddle points, which, together with the downward or upward search algorithms, is widely used to construct the solution landscapes that attracts increasing attentions due to its successful applications in various
fields such as revealing the mechanism of nucleation of quasicrystals \cite{Yin2020nucleation} and the excited states and
excitation mechanisms of the rotational Bose-Einstein condensates \cite{YinBE}. It was proved in \cite{YinSISC} that a linear stable
steady state of HiSD of index-$k$ is exactly an index-$k$ saddle point. To ensure that the numerical scheme of HiSD also converges to the same target saddle point of HiSD, it is critical to ensure that the discrete HiSD evolves along the dynamical pathway of continuous HiSD. If we only analyze the discrete HiSD without considering its difference from the continuous HiSD, there is no guarantee for the discrete HiSD to converge to the same target saddle point of HiSD. For this reason, it is important to perform error estimates for a certain numerical scheme of HiSD to ensure its dynamical pathway convergence.    }

There are some recent progresses on the numerical analysis of HiSD \cite{Z3}. In \cite{Z3}, an explicit difference scheme was proposed for the HiSD and rigorous error estimates were performed to show its first-order accuracy, which ensured the dynamical convergence of numerical solutions to the saddle dynamics and consequently supported the accurate construction of the solution landscape \cite{HanYin,YinSCM}. Since the explicit Euler scheme is adopted in \cite{Z3}, the orthonormal property of the eigenvectors at each time step has been fully employed in simplifying the derivations and analyzing the properties of the numerical solutions.
While the explicit scheme is convenient to implement, it could suffer numerical instability if the time step size is large, and thus reduces the efficiency of the saddle dynamics.

To relax the time step size and accelerate the convergence behavior, we investigate a semi-implicit scheme of HiSD. Compared with the explicit scheme, the orthonormal property of the eigenvectors at each time step could not be fully employed due to  the semi-implicit treatment, which makes the numerical analysis intricate. Furthermore, we adopt a novel computational strategy by computing the eigenvectors $\{v_i\}_{i=1}^k$ and performing the orthonormalization procedure successively for $1\leq i\leq k$,
and such coupling scheme of the eigenvectors and the orthonormalization requires more technical and complicated analysis.
All these difficulties invalidate the delicate error estimates of HiSD in \cite{Z3}. Therefore, the improved analysis such as the technical splittings and multi-variable circulating induction procedure are required to prove the error estimate of the semi-implicit scheme. Furthermore, the generalized minimum residual (GMRES) method is applied to solve the semi-implicit system, the convergence of which needs to be analyzed that is not encountered in the explicit scheme.

Motivated by these discussions, we show the error estimate of the semi-implicit scheme of HiSD and the convergence of the GMRES method. We perform various numerical experiments to show the efficiency and accuracy of the semi-implicit scheme compared with the explicit scheme. The rest of the paper is organized as follows: In Section 2 we introduce the formulation of HiSD and its semi-implicit numerical scheme, as well as proving some properties of the numerical solutions. In Section 3 we develop novel techniques to prove a key estimate to support the error estimate of the semi-implicit scheme in Section 4. In Section 5 we prove the convergence of the GMRES solver for the semi-implicit system. Numerical experiments are performed in Section 6 to substantiate the theoretical findings and compare the explicit and the semi-implicit schemes. We finally address concluding remarks in Section 7.

\section{HiSD and semi-implicit scheme}\label{sec2}
Let $E(x)$ be the energy function with $x\in\mathbb R^N$, and define $F(x)=-\nabla E(x)$ and $J(x)=-\nabla^2 E(x)$. It is clear that $J(x)=J(x)^\top$. Then the saddle dynamics for an index-$k$ saddle point of $E(x)$ with $1\leq k\in\mathbb N$ reads \cite{YinSISC}
\begin{equation}\label{sadk}
	\left\{
	\begin{array}{l}
		\ds \frac{dx}{dt} =\beta\bigg(I -2\sum_{j=1}^k v_jv_j^\top \bigg)F(x),\\[0.075in]
		\ds \frac{dv_i}{dt}=\gamma \bigg( I-v_iv_i^\top-2\sum_{j=1}^{i-1}v_jv_j^\top\bigg)J(x)v_i,~~1\leq i\leq k
	\end{array}
	\right.
\end{equation}
where $x$ represents a position variable, $v_i (i=1,..., k)$ are $k$ directional variables, and $\beta$, $\gamma>0$ are relaxation parameters. It is shown in \cite{YinSISC} that if $\{v_i(0)\}_{i=1}^k$ are orthonormal vectors, then $\{v_i(t)\}_{i=1}^k$ are orthonormal for any $t>0$.
Throughout the paper we use $Q$ to denote a generic positive constant that may assume different values at different occurrences, and make the following assumptions:

\noindent\textbf{Assumption A:} The $F(x)$ could be represented as a sum of the linear part $\mathcal Lx$ for some matrix $\mathcal L$ and the nonlinear part $\mathcal N(x)$, that is, $F(x)=\mathcal Lx+\mathcal N(x)$, and there exists a constant $L>0$ such that the following linearly growth and Lipschitz conditions hold under the standard $l^2$ norm $\|\cdot\|$ of a vector or a matrix
$$\begin{array}{c}
	\ds \max\{\|J(x_2)-J(x_1)\|,\|\mathcal Lx_2-\mathcal Lx_1\|,\|\mathcal N(x_2)-\mathcal N(x_1)\|\}\leq L\|x_2-x_1\|,\\[0.1in]
	\ds\max\{\|\mathcal Lx\|,\|\mathcal N(x)\|\}\leq L(1+\|x\|),~~x,x_1,x_2\in \mathbb R^N.
\end{array}  $$
\begin{remark}
	In Assumption A the $F$ is decomposed as a combination of the linear and nonlinear parts such that in the following semi-implicit scheme, the linear part will be treated implicitly while the nonlinear part will be treated explicitly. This semi-implicit treatment is indeed consistent with the commonly-used semi-implicit numerical methods for some nonlinear problems such as the phase-field equation \cite{CheShe} to ensure the computational stability and accuracy.
\end{remark}

Based on the Assumption A, it is shown in \cite{Z3} that $\|x(t)\|$ is bounded for $t\in [0,T]$ for a given terminal time $T$.

\subsection{Semi-implicit scheme}\label{sec21}
We consider the semi-implicit scheme of the index-$k$ saddle dynamics (\ref{sadk}) on the time interval $[0,T]$ equipped with the initial conditions
\begin{equation}\label{inicondk}
	\ds x(0)=x_0,~~v_i(0)=v_{i,0},~~v_{i,0}^\top v_{j,0}=\delta_{i,j},~~1\leq i,j\leq k.
\end{equation}
Let $0=t_0<t_1<\cdots t_N=T$ be the uniform partition of $[0,T]$ with the step size $\tau=T/N$, and let $\{x_n,v_{i,n}\}_{n=0}^N$ be the numerical solution of (\ref{sadk}). Then we discretize the first-order derivative by the Euler scheme and treat the linear and nonlinear parts on the right-hand side of (\ref{sadk}) via the implicit and explicit manner, respectively, to obtain the semi-implicit scheme of (\ref{sadk}) for $1\leq n\leq N$ as follows
\begin{equation}\label{FDsadk}
	\left\{
	\begin{array}{l}
		\ds x_{n} =x_{n-1}+\tau\beta\bigg(I -2\sum_{j=1}^k v_{j,n-1}v_{j,n-1}^\top \bigg)\big(\mathcal Lx_n+\mathcal N(x_{n-1})\big),\\[0.075in]
		\left.
		\begin{array}{l}
			\ds \tilde v_{i,n}=v_{i,n-1}+\tau\gamma\bigg( I-2\sum_{j=1}^{i-1}v_{j,n}v_{j,n}^\top\bigg)J(x_{n})\tilde v_{i,n}\\[0.1in]
			\ds\hspace{1.5in}-\tau\gamma v_{i,n-1}v_{i,n-1}^\top J(x_{n})v_{i,n-1},\\
			\ds  v_{i,n}=\frac{1}{Y_{i,n}}\bigg(\ds\tilde v_{i,n}-\sum_{j=1}^{i-1}(\tilde v_{i,n}^\top v_{j,n})v_{j,n}\bigg),
		\end{array}
		\right\}
		~~ 1\leq i\leq k.
	\end{array}
	\right.
\end{equation}
Here the last equation represents the Gram-Schmidt orthonormalization in order to maintain the orthonormal property of the vectors as in the continuous case, and thus $Y_{i,n}$ represents the norm of the vector in $(\cdots)$, i.e.,
$$\begin{array}{rl}
	\ds Y_{i,n}:\hspace{-0.1in}&\ds=\bigg\|\tilde v_{i,n}-\sum_{j=1}^{i-1}(\tilde v_{i,n}^\top v_{j,n})v_{j,n}\bigg\|=\bigg(\|\tilde v_{i,n}\|^2-\sum_{j=1}^{i-1}(\tilde v_{i,n}^\top v_{j,n})^2\bigg)^{1/2}.
\end{array} $$

{\color{blue}
\begin{remark}
The purpose of the Gram-Schmidt orthonormalization aims to ensure the orthonormality of directional vectors as in the continuous HiSD, that is, ensure the Stiefel manifold constraint \cite{Haibook}. From the later point of view, one may also apply the projection method for differential equations on manifold \cite{Hai,Haibook} to retract the dynamics of directional vectors in (\ref{FDsadk}) to the Stiefel manifold at each time step, and the error estimate could be immediately obtained by directly applying the conclusions in \cite[Section IV.4]{Haibook}. However, according to \cite[Example 4.6]{Haibook}, computing the projection for the current case requires to perform the singular value decomposition for the matrix concatenated by $\{\tilde v_{i,1},\cdots,\tilde v_{i,n}\}$ at each time step $t_n$, which could be computationally expensive. Thus we instead adopt the easy-to-implement approach (i.e. the Gram-Schmidt process) in the scheme, which, however, is no longer the projection to the Stiefel manifold such that the conclusions in \cite[IV.4]{Haibook} could not be applied to reach error estimates that motivates the current study.
\end{remark}
}

By Assumption A and $\|v_{j,n-1}\|=1$, we could apply the discrete Gronwall inequality to derive from the first equation of (\ref{FDsadk}) that $\|x_n\|$ is bounded by some constant $Q_x$ for $1\leq n\leq N$, which, based on the scheme of $\tilde v_{i,n}$, implies the boundedness of $\|\tilde v_{i,n}\|$ for $\tau$ small enough. Note that the boundedness of $x_n$ and $\tilde v_{i,n}$ further implies
\begin{equation}\label{simp}
	\begin{array}{l}
		\|\tilde v_{i,n}-v_{i,n-1}\|\\
		\ds\quad=\bigg\|\tau\gamma\bigg( I-2\sum_{j=1}^{i-1}v_{j,n}v_{j,n}^\top\bigg)J(x_{n})\tilde v_{i,n}-\tau\gamma v_{i,n-1}v_{i,n-1}^\top J(x_{n})v_{i,n-1}\bigg\|\leq Q\tau.
	\end{array}
\end{equation}

Compared with the explicit scheme presented in \cite{Z3}
\begin{equation}\label{schemeexp}
	\left\{
	\begin{array}{l}
		\ds x_{n} =x_{n-1}+\tau\beta\bigg(I -2\sum_{j=1}^k v_{j,n-1}v_{j,n-1}^\top \bigg)F(x_{n-1}),\\[0.075in]
		\ds \tilde v_{i,n}=v_{i,n-1}+\tau\gamma\bigg( I-v_{i,n-1}v_{i,n-1}^\top\\
		\ds\hspace{1in}-2\sum_{j=1}^{i-1}v_{j,n-1}v_{j,n-1}^\top\bigg)J(x_{n-1})v_{i,n-1},~~1\leq i\leq k,\\
		\ds  v_{i,n}=\frac{1}{Y_{i,n}}\bigg(\ds\tilde v_{i,n}-\sum_{j=1}^{i-1}(\tilde v_{i,n}^\top v_{j,n})v_{j,n}\bigg),~~1\leq i\leq k,
	\end{array}
	\right.
\end{equation}
the schemes and the computational strategies have salient differences:
\begin{itemize}
	\item[(a)] In the explicit scheme (\ref{schemeexp}), all variables on the right-hand side of the equations take the values at the previous time step $t_{n-1}$. In this way, the orthonormal property of the vectors $\{v_{i,n-1}\}_{i=1}^k$ at the time step $t_{n-1}$ could be fully employed to facilitate the numerical analysis as performed in \cite{Z3,Zhang2023}. However, in (\ref{FDsadk}) the linear parts on the right-hand side of the equations of vectors are treated implicitly, for which the orthonormal property could not be applied directly and thus the error estimate could be significantly complicated.
	
	\item[(b)] In the explicit scheme (\ref{schemeexp}), the schemes of the vectors $\{\tilde v_{i,n}\}_{i=1}^k$ are firstly solved, and then their orthonormalization are independently performed. In the semi-implicit scheme (\ref{FDsadk}), the computational strategy is different in that the scheme of the vector $\tilde v_{i,n}$ in the second equation and its orthonormalization with $\{v_{j,n}\}_{j=1}^{i-1}$ in the third equation are solved consecutively for $i=1,\cdots,k$. In this way, the orthonormalized vectors $\{v_{j,n}\}_{j=1}^{i-1}$ at the current time step $t_n$ serve as inputs in the scheme of the vector $\tilde v_{i,n}$, which could be more appropriate than using the vectors at the previous time step in the explicit scheme. However, this computational strategy leads to the coupling of the schemes of $\{\tilde v_{i,n}\}$ and the orthonormalization procedure, which makes the numerical analysis more intricate.
\end{itemize}

Concerning these difficulties, we derive novel analysis methods to carry out error estimates in subsequent sections.

\subsection{Auxiliary estimates}\label{sec22}
We prove several auxiliary estimates to support the error estimates.
\begin{lemma}\label{lemm1}
	For $1\leq m<i\leq k$, there exist positive constants $Q_0$ and $Q_1$, which are independent from $m$, $i$ and $n$, such that the following estimate holds
	$$\|\tilde v_{i,n}^\top \tilde v_{m,n}\|\leq Q_0\tau\sum_{l=1}^m\|\tilde v_{l,n}-v_{l,n}\|+Q_1\tau^2. $$
\end{lemma}
\begin{proof}
	By the definitions of $\tilde v_{i,n}$ and $\tilde v_{m,n}$ we have
	\begin{equation}\label{AA}
		\begin{array}{l}
			\ds \tilde v_{i,n}^\top \tilde v_{m,n}=\tau\gamma\bigg(v_{m,n-1}^\top J(x_n)\tilde v_{i,n}-2\sum_{j=1}^{i-1}v_{m,n-1}^\top v_{j,n}v_{j,n}^\top J(x_n)\tilde v_{i,n}\\
			\ds\qquad+\tilde v_{m,n}^\top J(x_n)^\top v_{i,n-1}-2\sum_{j=1}^{m-1}v_{j,n}^\top v_{i,n-1} v_{j,n}^\top J(x_n)\tilde v_{m,n}\bigg)+O(\tau^2)\\
			\ds\qquad\qquad=:\sum_{l=1}^4 A_l+O(\tau^2).
		\end{array}
	\end{equation}
	By $v_{m,n}^\top v_{j,n}=\delta_{m,j}$ we rewrite $A_2$ as
	\begin{equation}\label{conc1} \begin{array}{l}
			\ds A_2=-2\tau\gamma\sum_{j=1}^{i-1}v_{m,n-1}^\top v_{j,n}v_{j,n}^\top J(x_n)\tilde v_{i,n}\\
			\ds\quad~
			=-2\tau\gamma\sum_{j=1}^{i-1}(v_{m,n-1}^\top-v_{m,n}^\top) v_{j,n}v_{j,n}^\top J(x_n)\tilde v_{i,n}-2\tau\gamma v_{m,n}^\top J(x_n) \tilde v_{i,n},
	\end{array}  \end{equation}
	which leads to
	\begin{equation}\label{conc2} \begin{array}{rl}
			\ds A_1+A_2+A_3\hspace{-0.1in}&\ds=\tau\gamma\big(v_{m,n-1}^\top J(x_n)\tilde v_{i,n}-v_{m,n}^\top J(x_n) \tilde v_{i,n} \big)\\[0.1in]
			&\ds\hspace{-0.5in}+\tau\gamma\big(\tilde v_{m,n}^\top J(x_n)^\top v_{i,n-1}-v_{m,n}^\top J(x_n) \tilde v_{i,n}\big)\\[0.1in]
			&\ds\hspace{-0.5in}-2\tau\gamma\sum_{j=1}^{i-1}(v_{m,n-1}^\top-v_{m,n}^\top) v_{j,n}v_{j,n}^\top J(x_n)\tilde v_{i,n}=:B_1+B_2+B_3.
	\end{array} \end{equation}
	By the splitting
	\begin{equation}\label{vsplit}
		v_{m,n-1}-v_{m,n}=(v_{m,n-1}-\tilde v_{m,n})+(\tilde v_{m,n}-v_{m,n})
	\end{equation}
	and (\ref{simp}),
	$B_1$ could be bounded as
	$$|B_1|=\tau\gamma|(v_{m,n-1}^\top -v_{m,n}^\top) J(x_n) \tilde v_{i,n}|\leq Q\tau^2+Q\tau \|\tilde v_{m,n}-v_{m,n}\|.$$
	In a similar manner we bound $B_3$ by
	$$\begin{array}{l}
		\ds |B_3|=2\tau\gamma\bigg|\sum_{j=1}^{i-1}(v_{m,n-1}^\top-v_{m,n}^\top) v_{j,n}v_{j,n}^\top J(x_n)\tilde v_{i,n}\bigg|\leq Q\tau^2+Q\tau \|\tilde v_{m,n}-v_{m,n}\|.
	\end{array} $$
	We then apply the symmetry of $J(x_n)$ to bound $B_2$ as
	$$\begin{array}{l}
		\ds|B_2|=\tau\gamma|(\tilde v_{m,n}^\top-v_{m,n}^\top) J(x_n) v_{i,n-1}+v_{m,n}^\top J(x_n) (v_{i,n-1}-\tilde v_{i,n})|\\[0.05in]\ds\quad\quad \leq Q\tau^2+Q\tau \|\tilde v_{m,n}-v_{m,n}\|.
	\end{array}   $$
	
	We finally bound $A_4$ as
	$$\begin{array}{rl}
		\ds|A_4|\hspace{-0.1in}&\ds=\bigg|2\tau\gamma\sum_{j=1}^{m-1}v_{j,n}^\top v_{i,n-1} v_{j,n}^\top J(x_n)\tilde v_{m,n}\bigg|\\[0.1in]
		&\ds=\bigg|2\tau\gamma\sum_{j=1}^{m-1}(v_{j,n}^\top-\tilde v_{j,n}^\top+\tilde v_{j,n}^\top-v_{j,n-1}^\top) v_{i,n-1} v_{j,n}^\top J(x_n)\tilde v_{m,n}\bigg|\\
		&\ds\leq Q\tau^2+Q\tau\sum_{j=1}^{m-1}\|v_{j,n}-\tilde v_{j,n}\|.
	\end{array}  $$
	We incorporate the preceding estimates to complete the proof.
\end{proof}

\begin{lemma}\label{lemm2}
	For $1\leq j\leq k$, there exist positive constants $Q_2$ and $Q_3$, which are independent from $j$ and $n$, such that the following estimate holds
	$$\big|\|\tilde v_{j,n}\|^2-1\big|\leq Q_2\tau\sum_{l=1}^{j-1}\|\tilde v_{l,n}-v_{l,n}\|+Q_3\tau^2. $$
\end{lemma}
\begin{proof}
	From the definition of $\tilde v_{j,n}$ we have
	$$\begin{array}{l}
		\ds \|\tilde v_{j,n}\|^2=1+2\tau\gamma\bigg(v_{j,n-1}^\top-2\sum_{l=1}^{j-1}v_{j,n-1}^\top v_{l,n} v_{l,n}^\top\bigg)J(x_n)\tilde v_{j,n}\\[0.15in]
		\ds\qquad\qquad\qquad-2\tau\gamma v_{j,n-1}^\top J(x_n)v_{j,n-1}+O(\tau^2),
	\end{array}  $$
	which, together with (\ref{simp}), implies
	$$\begin{array}{l}
		\ds\big|\|\tilde v_{j,n}\|^2-1\big|= \bigg|2\tau\gamma v_{j,n-1}^\top J(x_n)(\tilde v_{j,n}-v_{j,n-1}) \\[0.15in]
		\ds\qquad -4\tau\gamma\sum_{l=1}^{j-1}v_{j,n-1}^\top (v_{l,n}-\tilde v_{l,n}+\tilde v_{l,n}-v_{l,n-1}) v_{l,n}^\top J(x_n)\tilde v_{j,n}+O(\tau^2)\bigg|\\[0.15in]
		\ds\qquad\qquad\qquad\leq Q_2\tau\sum_{l=1}^{j-1}\|\tilde v_{l,n}-v_{l,n}\|+Q_3\tau^2.
	\end{array} $$
	Thus we complete the proof.
\end{proof}

\section{Estimate of $v_{j,n}-\tilde v_{j,n}$}\label{sec3}
In Section \ref{sec22} we could observe that the estimates of several quantities depend on the norm of $v_{j,n}-\tilde v_{j,n}$. However, we will find in the following derivations that the estimate of $v_{j,n}-\tilde v_{j,n}$ in turn depends on these quantities. To resolve this issue, we carry out a multi-variable circulating induction procedure (cf. the proof of Theorem \ref{lem2k}) to bound $v_{j,n}-\tilde v_{j,n}$ in this section.

For $\bar G>Q_3Q_4+kQ_1$ where $Q_1$ and $Q_3$ are introduced in Lemmas \ref{lemm1}-\ref{lemm2} and $Q_4>1$ represents the bound of $\{\tilde v_{j,n}\}_{j=1,n=0}^{k,N}$, there exists an intermediate constant $G>0$ such that
$$\bar G>Q_3Q_4+kG\text{ and } G>Q_1. $$
In particular, as $Q_4>1$, we have $\bar G>Q_3$. Then for $\tau$ small enough the following inequalities hold
\begin{equation}\label{GG} \begin{array}{c}
		\ds\frac{Q_0\tau k\bar G+Q_1+kG^2\tau^2}{(1-Q_2\tau^3 k\bar G-Q_3\tau^2-kG^2\tau^4)^{1/2}}\leq G,\\[0.15in]
		\ds \frac{Q_4(Q_2\tau k \bar G+Q_3+kG^2\tau^2)+kG}{(1-Q_2\tau^3 k\bar G-Q_3\tau^2-kG^2\tau^4)^{1/2}}\leq \bar G.
\end{array}  \end{equation}
In subsequent proofs, we always choose sufficiently small step size $\tau$ such that the condition (\ref{GG}) is satisfied.

To start the induction, we need to estimate the properties of the numerical solutions at the first few steps.
\begin{lemma}\label{lemm3}
	Under the condition (\ref{GG}), the following estimates hold for $1\leq n\leq N$
	$$|\tilde v_{i,n}^\top v_{1,n}|\leq G\tau^2,~~1<i\leq k;~~\|\tilde v_{j,n}-v_{j,n}\|\leq \bar G\tau^2,~~1\leq j\leq 2.  $$
\end{lemma}
\begin{proof}
	We apply Lemma \ref{lemm2} to obtain
	$$\|\tilde v_{1,n}-v_{1,n}\|=\bigg\|\frac{\tilde v_{1,n}}{\|\tilde v_{1,n}\|}(\|\tilde v_{1,n}\|-1)\bigg\|\leq \big|\|\tilde v_{1,n}\|^2-1\big|\leq Q_3\tau^2\leq \bar G\tau^2. $$
	Then for $1< i\leq k$, we apply Lemma \ref{lemm1} and (\ref{GG}) to obtain
	$$\begin{array}{l}
		\ds|\tilde v_{i,n}^\top v_{1,n}|=\frac{|\tilde v_{i,n}^\top \tilde v_{1,n}|}{(\|\tilde v_{1,n}\|^2)^{1/2}}\leq \frac{Q_0\tau\|\tilde v_{1,n}-v_{1,n}\|+Q_1\tau^2}{(1-Q_3\tau^2)^{1/2}}\\[0.15in]
		\qquad\qquad \ds\leq \frac{Q_0\bar G\tau^3+Q_1\tau^2}{(1-Q_3\tau^2)^{1/2}}= \frac{Q_0\bar G\tau+Q_1}{(1-Q_3\tau^2)^{1/2}}\tau^2\leq G\tau^2.
	\end{array}   $$
	We incorporate the above two estimates and Lemma \ref{lemm2} to bound $Y_{2,n}$ as
	$$\begin{array}{rl}
		\ds Y_{2,n}\hspace{-0.1in}&\ds=\big(\|\tilde v_{2,n}\|^2-(\tilde v_{2,n}^\top v_{1,n})^2 \big)^{1/2}\\[0.05in]
		&\ds\in \big(1\pm (Q_2\tau\|\tilde v_{1,n}-v_{1,n}\|+Q_3\tau^2+G^2\tau^4) \big)^{1/2}\\[0.05in]
		&\ds\in \big(1\pm (Q_2\bar G\tau^3+Q_3\tau^2+G^2\tau^4) \big)^{1/2},
	\end{array}  $$
	which further implies
	$$|Y_{2,n}-1|\leq |Y_{2,n}^2-1|\leq Q_2\bar G\tau^3+Q_3\tau^2+G^2\tau^4.  $$
	We invoke these estimates into the expression of $v_{2,n}-\tilde v_{2,n}$ to get
	$$\begin{array}{rl}
		\ds\|v_{2,n}-\tilde v_{2,n}\|\hspace{-0.1in}&\ds=\frac{1}{Y_{2,n}}\|(1-Y_{2,n})\tilde v_{2,n}-(\tilde v_{2,n}^\top v_{1,n})v_{1,n} \|\\[0.15in]
		&\ds\leq \frac{Q_4(Q_2\bar G\tau^3+Q_3\tau^2+G^2\tau^4)+G\tau^2}{\big(1- (Q_2\bar G\tau^3+Q_3\tau^2+G^2\tau^4) \big)^{1/2}}\\[0.15in]
		&\ds= \frac{Q_4(Q_2\bar G\tau+Q_3+G^2\tau^2)+G}{\big(1- (Q_2\bar G\tau^3+Q_3\tau^2+G^2\tau^4) \big)^{1/2}}\tau^2\leq \bar G\tau^2.
	\end{array}  $$
	Thus we complete the proof.
\end{proof}

\begin{theorem}\label{lem2k}
	Under the condition (\ref{GG}), the following estimate holds for $1\leq n\leq N$
	$$\|v_{i,n}-\tilde v_{i,n}\|\leq \bar G\tau^2,~~1\leq i\leq k.$$
\end{theorem}
\begin{proof}
	We prove this theorem by induction for the following two relations
	\begin{equation*}
		\begin{array}{l}
			\ds (\mathbb A):~ \max_{m<i\leq k}\|\tilde v_{i,n}^\top v_{m,n}\|\leq G\tau^2 \text{ for some }1\leq m\leq k-1;\\[0.05in]
			\ds(\mathbb B):~\|v_{j,n}-\tilde v_{j,n}\|\leq \bar G\tau^2\text{ for some }1\leq j\leq k.
		\end{array}
	\end{equation*}
	Lemma \ref{lemm3} implies that $(\mathbb A)$ holds for $m=1$ and $(\mathbb B)$ holds for $1\leq j\leq 2$. Suppose
	\begin{equation}\label{induc}
		(\mathbb A) \text{ holds for }1\leq m\leq m^*-1\text{ and }(\mathbb B)\text{ holds for }1\leq j\leq m^*
	\end{equation}
	for some $1\leq m^*<k-1$. Then we remain to show that
	$$
	(\mathbb A) \text{ holds for } m= m^*\text{ and }(\mathbb B)\text{ holds for } j= m^*+1
	$$
	for the sake of mathematical induction.
	
	We use Lemma \ref{lemm2} and $(\mathbb A)$ with $1\leq m\leq m^*-1$ ( by induction hypotheses (\ref{induc})) to bound $Y_{m^*,n}$ by
	\begin{equation}\label{bnd0Y}
		\begin{array}{rl}
			\ds Y_{m^*,n}\hspace{-0.1in}&\ds=\bigg(\|\tilde v_{m^*,n}\|^2-\sum_{j=1}^{m^*-1}(\tilde v_{m^*,n}^\top v_{j,n})^2\bigg)^{1/2}\\
			& \ds\in \bigg[1\pm \bigg(Q_2\tau\sum_{l=1}^{m^*-1}\|\tilde v_{l,n}-v_{l,n}\|+Q_3\tau^2+(m^*-1) G^2\tau^4\bigg)\bigg]^{1/2} \\[0.15in]
			& \ds\in \big[1\pm \big(Q_2 (m^*-1)\bar G\tau^3+Q_3\tau^2+(m^*-1) G^2\tau^4\big)\big]^{1/2} ,
		\end{array}
	\end{equation}
	where in the above equation $Y_{m^*,n}\in [1\pm (\cdots)]^{1/2}$ means $[1- (\cdots)]^{1/2}\leq Y_{m^*,n}\leq [1+ (\cdots)]^{1/2}$.
	We then invoke the induction hypotheses (\ref{induc}), (\ref{bnd0Y}), the condition (\ref{GG}) and Lemmas \ref{lemm1}-\ref{lemm2} into the expression of $\tilde v_{i,n}^\top v_{m^*,n}$ to obtain for $m^*<i\leq k$
	$$\begin{array}{rl}
		\ds | \tilde v_{i,n}^\top v_{m^*,n}|\hspace{-0.1in}&\ds=\frac{1}{Y_{m^*,n}}\bigg|\tilde v_{i,n}^\top\tilde v_{m^*,n}-\sum_{j=1}^{m^*-1}(\tilde v_{m^*,n}^\top v_{j,n})(\tilde v_{i,n}^\top v_{j,n})\bigg| \\[0.15in]
		&\ds\leq\frac{1}{Y_{m^*,n}}\bigg( Q_0\tau\sum_{l=1}^{m^*}\|\tilde v_{l,n}-v_{l,n}\|+Q_1\tau^2+(m^*-1)G^2\tau^4\bigg) \\[0.2in]
		&\ds \leq \frac{Q_0\tau m^*\bar G+Q_1+(m^*-1)G^2\tau^2}{(1-Q_2\tau^3 (m^*-1)\bar G-Q_3\tau^2-(m^*-1)G^2\tau^4)^{1/2}}\tau^2\leq G\tau^2,
	\end{array}$$
	which implies that $(\mathbb A)$ holds for $m=m^*$. We then use Lemma \ref{lemm2} and $(\mathbb A)$ with $1\leq m\leq m^*$ to bound $Y_{m^*+1,n}$ by
	\begin{equation}\label{bndY}
		\begin{array}{rl}
			\ds Y_{m^*+1,n}\hspace{-0.1in}&\ds=\bigg(\|\tilde v_{m^*+1,n}\|^2-\sum_{j=1}^{m^*}(\tilde v_{m^*+1,n}^\top v_{j,n})^2\bigg)^{1/2}\\
			& \ds\in \bigg[1\pm \bigg(Q_2\tau\sum_{l=1}^{m^*}\|\tilde v_{l,n}-v_{l,n}\|+Q_3\tau^2+m^* G^2\tau^4\bigg)\bigg]^{1/2} \\[0.15in]
			& \ds\in \big[1\pm \big(Q_2 m^*\bar G\tau^3+Q_3\tau^2+m^* G^2\tau^4\big)\big]^{1/2} ,
		\end{array}
	\end{equation}
	which implies
	$$|1-Y_{m^*+1,n}|\leq |1-Y_{m^*+1,n}^2|\leq Q_2 m^*\bar G\tau^3+Q_3\tau^2+m^* G^2\tau^4. $$
	We invoke this and $(\mathbb A)$ with $1\leq m\leq m^*$ in $v_{m^*+1,n}-\tilde v_{m^*+1,n}$ to get
	\begin{equation}\label{diff}
		\begin{array}{l}
			\ds \|v_{m^*+1,n}-\tilde v_{m^*+1,n}\|\\[0.05in]
			\ds\qquad=\frac{1}{Y_{m^*+1,n}}\bigg\|(1-Y_{m^*+1,n})\tilde v_{m^*+1,n}-\sum_{j=1}^{m^*}(\tilde v_{m^*+1,n}^\top v_{j,n})v_{j,n}\bigg\|\\[0.2in]
			\ds\qquad\leq \frac{Q_4(Q_2\tau m^* \bar G+Q_3+m^*G^2\tau^2)+m^*G}{(1-Q_2\tau^3 m^*\bar G-Q_3\tau^2-m^*G^2\tau^4)^{1/2}}\tau^2\leq \bar G\tau^2,
		\end{array}
	\end{equation}
	which implies that $(\mathbb B)$ holds for $j=m^*+1$ and thus completes the proof. \end{proof}

\section{Error estimates}\label{sec4}
We prove error estimates for the semi-implicit scheme (\ref{FDsadk}). Define the errors
$$e^x_n:=x(t_n)-x_n,~~e^{v_i}_n:=v_i(t_n)-v_{i,n},~~1\leq n\leq N,~~1\leq i\leq k.$$
To bound $e^x_n$, we derive the reference equation from the first equation of (\ref{sadk}) via the forward Euler discretization
$$x(t_{n}) =x(t_{n-1})+\tau\beta\bigg(I -2\sum_{j=1}^k v_j(t_{n-1})v_j(t_{n-1})^\top \bigg)F(x(t_{n-1}))+O(\tau^2). $$
Then we subtract the scheme of $x_n$ in (\ref{FDsadk}) from this equation and apply almost the same derivation as \cite[Equation 4.14]{Z3} to obtain
\begin{equation}\label{xref}
	\|e^x_n\|\leq Q\tau\sum_{m=1}^{n-1}\sum_{j=1}^k \|e^{v_j}_m\|+Q\tau,~~1\leq n\leq N.
\end{equation}
We observe that $\|e^x_n\|$ is bounded in terms of $\|e^{v_j}_n\|$, which is estimated in the following theorem.

\begin{theorem}\label{thmevk}
	Under the Assumption A, the following estimate holds for the semi-implicit scheme (\ref{FDsadk}) for $\tau$ sufficiently small
	$$\|e^x_{n}\|+\sum_{i=1}^k \|e^{v_i}_n\|\leq Q\tau,~~1\leq n\leq N. $$
	Here $Q$ is independent from $\tau$, $n$ and $N$.
\end{theorem}
\begin{proof}
	we derive the reference equation from the second equation of (\ref{sadk}) via the backward Euler discretization for $1\leq i\leq k$
	\begin{equation}\label{Refk}
		\begin{array}{l}
			\ds v_i(t_{n})=v_i(t_{n-1})+\tau\gamma\bigg( I-v_i(t_{n})v_i(t_{n})^\top\\
			\ds\hspace{0.5in}-2\sum_{j=1}^{i-1}v_j(t_{n})v_j(t_{n})^\top\bigg)J(x(t_{n}))v_i(t_{n})+O(\tau^2)\\
			\ds\qquad~\,=v_i(t_{n-1})+\tau\gamma\bigg( I-2\sum_{j=1}^{i-1}v_j(t_{n})v_j(t_{n})^\top\bigg)J(x(t_{n}))v_i(t_{n})\\[0.2in]
			\ds\hspace{0.5in}-\tau\gamma v_i(t_{n-1})v_i(t_{n-1})^\top J(x(t_{n}))v_i(t_{n-1})+C_n
		\end{array}
	\end{equation}
	where
	$$\begin{array}{l}
		\ds C_n=\tau\gamma\big(v_i(t_{n-1})v_i(t_{n-1})^\top J(x(t_{n}))v_i(t_{n-1})\\[0.05in]
		\ds\qquad\qquad -v_i(t_{n})v_i(t_{n})^\top J(x(t_{n}))v_i(t_{n})\big)+O(\tau^2).
	\end{array}
	$$
	We then substitute $\tilde v_{i,n}$ by $v_{i,n}-(v_{i,n}-\tilde v_{i,n})$ in the scheme of $\tilde v_{i,n}$ in (\ref{FDsadk}) to obtain
	$$ v_{i,n}=v_{i,n-1}+\tau\gamma\bigg( I-2\sum_{j=1}^{i-1}v_{j,n}v_{j,n}^\top\bigg)J(x_{n}) v_{i,n}-\tau\gamma v_{i,n-1}v_{i,n-1}^\top J(x_{n})v_{i,n-1}+D_n
	$$
	where
	$$
	D_n=-\tau\gamma\bigg( I-2\sum_{j=1}^{i-1}v_{j,n}v_{j,n}^\top\bigg)J(x_{n})(v_{i,n}-\tilde v_{i,n})+(v_{i,n}-\tilde v_{i,n}). $$
	We subtract this equation from (\ref{Refk}) to obtain
	\begin{equation}\label{exx}
		\begin{array}{rl}
			\ds e^{v_i}_n\hspace{-0.1in}&\ds=e^{v_i}_{n-1}+\tau\gamma\big(J(x(t_{n}))v_i(t_{n})-J(x_{n})v_{i,n}\big)\\[0.05in]
			&\ds-\tau\gamma \big( v_i(t_{n-1})v_i(t_{n-1})^\top J(x(t_{n}))v_i(t_{n-1})-v_{i,n-1}v_{i,n-1}^\top J(x_{n})v_{i,n-1}\big)\\[0.05in]
			&\ds-2\tau\gamma \sum_{j=1}^{i-1}\big[ v_j(t_{n})v_j(t_{n})^\top J(x(t_{n}))v_i(t_{n})-v_{j,n}v_{j,n}^\top J(x_{n})v_{i,n}\big]+C_n-D_n.
		\end{array}
	\end{equation}
	We introduce intermediate terms to bound the third right-hand side term of (\ref{exx}) as
	$$\begin{array}{l}
		\ds \|v_i(t_{n-1})v_i(t_{n-1})^\top J(x(t_{n}))v_i(t_{n-1})-v_{i,n-1}v_{i,n-1}^\top J(x_{n})v_{i,n-1}\|\\[0.05in]
		\ds\quad=\|e^{v_i}_{n-1}v_i(t_{n-1})^\top J(x(t_{n}))v_i(t_{n-1})+v_{i,n-1}(e^{v_i}_{n-1})^\top J(x(t_{n}))v_i(t_{n-1})\\[0.05in]
		\ds\quad\quad+v_{i,n-1}v_{i,n-1}^\top (J(x(t_n))-J(x_{n}))v_i(t_{n-1})+v_{i,n-1}v_{i,n-1}^\top J(x_{n})e^{v_i}_{n-1}\|\\[0.05in]
		\ds\quad\leq Q\|e^{v_i}_{n-1}\|+Q\|e^x_n\|.
	\end{array} $$
	We apply Theorem \ref{lem2k} to bound $D_n$ as
	$$\|D_n\|\leq Q\tau^2. $$
	The other right-hand side terms of (\ref{exx}) could be bounded similarly, leading to
	\begin{equation*}
		\begin{array}{rl}
			\ds\| e^{v_i}_n\|\hspace{-0.1in}&\ds\leq \|e^{v_i}_{n-1}\|+Q\tau\big(\|e^x_{n}\|+\|e^{v_i}_n\|+\|e^{v_i}_{n-1}\|\big)+Q\tau\sum_{j=1}^{i-1}\|e^{v_j}_n\| +Q\tau^2\\
			\ds&\ds\leq \|e^{v_i}_{n-1}\|+Q\tau\big(\|e^{v_i}_n\|+\|e^{v_i}_{n-1}\|\big)+Q\tau^2\sum_{m=1}^{n-1}\sum_{j=1}^k\|e^{v_j}_m\|\\
			&\ds\qquad+Q\tau\sum_{j=1}^{i-1}\|e^{v_j}_n\| +Q\tau^2.
		\end{array}
	\end{equation*}
	Adding this equation from $i=1$ to $k$ and denoting
	$R_{n}:=\|e^{v_1}_n\|+\cdots+\|e^{v_k}_n\| $
	yield
	\begin{equation}\label{mh8k}
		\ds R_{n}\leq R_{n-1}+Q\tau ( R_{n-1}+R_n)+Q\tau^2\sum_{m=1}^{n-1} R_{m} +Q\tau^2.
	\end{equation}
	Adding this equation from $n=1$ to $n_*$ and using
	$$\tau^2\sum_{n=1}^{n_*}\sum_{m=1}^{n-1}R_{m}=\tau^2\sum_{m=1}^{n_*-1}\sum_{n=m+1}^{n_*}R_{m}\leq T \tau\sum_{m=1}^{n_*-1}R_{m}$$
	and $R_0=0$ we get
	\begin{equation}\label{mh9k}
		\ds R_{n_*}\leq Q\tau\sum_{n=1}^{n_*}R_{n} +Q\tau,
	\end{equation}
	which implies
	$$
	\ds R_{n_*}\leq\frac{1}{1-Q\tau}\bigg( Q\tau\sum_{n=1}^{n_*-1}R_{n} +Q\tau\bigg)\leq Q\tau\sum_{n=1}^{n_*-1}R_{n} +Q\tau.
	$$
	Then an application of the discrete Gronwall inequality leads to
	$$\ds R_{n}\leq Q\tau,~~1\leq n\leq N. $$
	Plugging this estimate back to (\ref{xref}) yields
	\begin{equation*}
		\|e^x_n\|\leq Q\tau\sum_{m=1}^{n-1}R_m+Q\tau\leq Q\tau,
	\end{equation*}
	which completes the proof.
\end{proof}

\section{Convergence of GMRES solver}
For implementation, we apply the GMRES method to solve the semi-implicit system (\ref{FDsadk}). To prove the convergence rate of the GMRES solver, we reformulate the semi-implicit schemes in (\ref{FDsadk}) into the following form
\begin{equation}
	\left\{
	\begin{array}{ll}
		G x_n= a,&\\[0.05in]
		H_i\tilde{v}_{i,n} = b_i,& 2\leq i\leq k,
	\end{array}
	\right.
	\label{linear_sys}
\end{equation}
where
\begin{equation}
	\begin{aligned}
		&G=I - \tau\beta\bigg(I-2\sum_{j=1}^kv_{j,n-1}v_{j,n-1}^\top\bigg)\mathcal L,\\
		&a = x_{n-1} +\tau\beta \bigg(I-2\sum_{j=1}^kv_{j,n-1}v_{j,n-1}^\top\bigg)\mathcal N(x_{n-1}),\\
		&H_i = I - \tau\gamma\bigg(I-2\sum_{j=1}^{i-1}v_{j,n}v_{j,n}^{\top}\bigg)J(x_n),\\ &b_i=v_{i,n-1} - \tau\gamma(v^\top_{i,n-1}J(x_n)v_{i,n-1})v_{i,n-1}.
	\end{aligned}
\end{equation}	
To prove the desired results, we refer the following convergence estimate for the GMRES method \cite{trefethen1997numerical}.
\begin{theorem}
	At the step $m$ of the GMRES iteration on solving the linear system $Ax=b$, where $b\in\mathbb{R}^N$, $A\in\mathbb{R}^{N\times N}$ and $x^{(m)}$ refers to the numerical solution at this step, the following estimate holds for the residual $r_m=b-Ax^{(m)}$
	\begin{equation}
		\frac{\|r_m\|}{\|b\|} \leq \inf_{p_m\in P_m}\|p_m(A)\|,
	\end{equation}
	where $P_m=\{p(z)=\sum_{i=0}^m a_i z^i|p(0)=1, a_i\in\mathbb{R}, 0\leq i\leq m
	\}$.
	\label{gmres}
\end{theorem}

Based on the convergence result, we derive the convergence rate of the GMRES for solving (\ref{linear_sys}) in the following theorem.
\begin{theorem}
	If the time step size $\tau$ satisfies
	\[ \tau \leq \frac{q}{\max\{\gamma\max_{\|x\|\leq Q_x }\|J(x)\|, \beta \|\mathcal L\|\}},\quad q\in(0,1),\]
	where $Q_x$ is the bound of $\|x_n\|$ for $1\leq n\leq N$ introduced above (\ref{simp}), then the GMRES method for solving (\ref{linear_sys}) converges for $m\geq 0$
	\[\frac{\|r_{i,m}\|}{\|b_i\|} \leq q^m\mbox{ for }2\leq i\leq k,\quad \frac{\|r_m\|}{\|a\|}\leq q^m,\]
	where $\tilde{v}_{i,n}^{(m)}$ and $x_n^{(m)}$ refer to the GMRES numerical solutions of $\tilde{v}_{i,n}$ and $x_n$, respectively, at the step $m$ and the residuals are defined as $r_{i,m} = b_i - H_i\tilde{v}_{i,n}^{(m)}$ for $1\leq i\leq k$ and $r_m = a - Gx_n^{(m)}$.
	\label{linear_v}
\end{theorem}
\begin{remark}
	This theorem indicates that under suitable selection of the time step size $\tau$, the GMRES solver converges rapidly such that only a few iterations are required in practice.
\end{remark}
\begin{proof}
	We first consider the system $H_i\tilde{v}_{i,n} = b_i$. We select the polynomial $p_m(z)=(1-z)^m$ in Theorem \ref{gmres} to obtain
	\begin{equation}
		\frac{\|r_{i,m}\|}{\|b_i\|}\leq \|(I-H_i)^m\|\leq \|I-H_i\|^m.
	\end{equation}
	Since $\{v_j^{(n)}\}_{j=1}^i$ are orthonormal vectors, $I-2\sum_{j=1}^{i-1}v_{j,n}v_{j,n}^{\top}$ preserves the $l^2$ norm such that
	\begin{equation}
		\|I-{H}_i\|\leq \tau\gamma\bigg\|I-2\sum_{j=1}^{i-1}v_{j,n}v_{j,n}^{\top}\bigg\|\|J(x_n)\| = \tau\gamma\|J(x_n)\|.
	\end{equation}
	We incorporate the above two equations to get
	\begin{equation}
		\frac{\|r_{i,m}\|}{\|b_i\|}\leq (\tau\gamma\|J(x_n)\|)^m.
	\end{equation}
	Thus if we choose the time step size $\tau \leq\frac{q}{\gamma\|J(x_n)\|}$ for some $q\in(0,1)$, we reach the convergence estimate $\frac{\|r_{i,m}\|}{\|b_i\|}\leq q^m$. Similarly, for the GMRES solver on the system $Gx_n = a$ we have the following convergence estimate based on Theorem \ref{gmres}
	\begin{equation}
		\frac{\|r_m\|}{\|a\|}\leq \|(I-G)^m\|\leq \|I-G\|^m\leq (\tau\beta\|\mathcal L\|)^m.
	\end{equation}
	If we set $\tau\leq \frac{q}{\beta\|\mathcal L\|_2}$ for some $q\in(0,1)$, we get the estimate of $r_m$ in this theorem. We incorporate the above two convergence results to complete the proof.
\end{proof}

\section{Numerical experiments}
In this section, we carry out numerical experiments to test the convergence rate (denoted by ``CR'' in tables) of the semi-implicit numerical scheme (\ref{FDsadk}) and compare the behavior of the semi-implicit (denoted by ``SI'') and explicit (denoted by ``EX'') schemes in computing the saddle points and constructing the solution landscapes. In practice, we adopt the dimer method \cite{Dimer} with the dimer length $l>0$ to efficiently evaluate the product of $J(x)$ and a vector $v$ as follows
$$J(x)v= \frac{F(x+lv)-F(x-lv)}{2l}+O(l^2)\approx \frac{F(x+lv)-F(x-lv)}{2l}. $$
It is worth mentioning that the numerical analysis results in previous sections still hold true if we substitute the product of $J(x)$ and the vector in the numerical scheme by its dimer approximation since the dimer length could be chosen as $O(\tau)$ such that the reminder $O(l^2)$ of the dimer approximation is indeed the high-order perturbation $O(\tau^2)$. 
\subsection{Accuracy tests}
We consider the saddle dynamics for the Eckhardt surface \cite{Eck}
$$\begin{array}{l}
	\ds E(x_1,x_2)= \text{exp}(-x_1^2-(x_2+1)^2)\\
	\ds\qquad\qquad\qquad+\text{exp}(-x_1^2-(x_2-1)^2)+4 \text{exp}\bigg(-3\frac{x_1^2+x_2^2}{2}\bigg)+\frac{x_2^2}{2}
\end{array} $$
and compute its index-1 saddle point with the initial conditions
$$x_0=
\left[\!\!
\begin{array}{c}
	\ds -3\\
	\ds 2
\end{array}
\!\!\right],~~
v_{1,0}=\frac{1}{\sqrt{5}}
\left[\!\!
\begin{array}{c}
	\ds -1\\
	\ds 2
\end{array}
\!\!\right]
$$
and index-2 saddle point with the initial conditions
$$x_0=
\left[\!\!
\begin{array}{c}
	\ds -3\\
	\ds 2
\end{array}
\!\!\right],~~
v_{1,0}=\frac{1}{\sqrt{5}}
\left[\!\!
\begin{array}{c}
	\ds -1\\
	\ds 2
\end{array}
\!\!\right],~~
v_{2,0}=\frac{1}{\sqrt{5}}
\left[\!\!
\begin{array}{c}
	\ds 2\\
	\ds 1
\end{array}
\!\!\right].
$$
As the exact solutions to the high-index saddle dynamics are not available, numerical solutions computed under $\tau=2^{-13}$ serve as the reference solutions. We set $\beta=\gamma=T=1$ for simplicity. Numerical results are presented in Tables \ref{table1:1}-\ref{table1:2}, which demonstrate the first-order accuracy of the semi-implicit scheme (\ref{FDsadk}) as proved in Theorem \ref{thmevk}.

\begin{table}[h]
	\setlength{\abovecaptionskip}{0pt}
	\centering
	\caption{Convergence rates of computing index-1 saddle point of Eckhardt surface.}
	\begin{tabular}{c|cc|cc} \cline{1-5}
		$\tau$& $\max_n \|e^x_n\|$ & CR &  $\max_n \|e^{v_1}_n\|$ &CR\\
		\cline{1-5}		
		1/32&	1.62E-02&		&3.55E-03	&\\
		1/64&	8.02E-03&	1.01&	1.75E-03&	1.02\\
		1/128&	3.97E-03&	1.02&	8.65E-04&	1.02\\
		1/256&	1.95E-03&	1.03&	4.25E-04&	1.03\\
		\hline
	\end{tabular}
	\label{table1:1}
\end{table}

\begin{table}[h]
\setlength{\abovecaptionskip}{0pt}
\centering
\caption{Convergence rates of computing index-2 saddle point of Eckhardt surface.}
\vspace{-0.25em}	
\begin{tabular}{c|cc|cc|cc} \cline{1-7}
	$\tau$& $\max_n \|e^x_n\|$ & CR &  $\max_n \|e^{v_1}_n\|$ &CR& $\max_n \|e^{v_2}_n\|$ &CR\\ \cline{1-7}		
	1/32&	1.16E-02&		&2.83E-03	&	&2.83E-03	&\\
	1/64&	5.74E-03&	1.02&	1.40E-03&	1.02&	1.40E-03&	1.02\\
	1/128&	2.84E-03&	1.02&	6.91E-04&	1.02&	6.91E-04&	1.02\\
	1/256&	1.39E-03&	1.03&	3.39E-04&	1.03&	3.39E-04&	1.03\\
	\hline
\end{tabular}
\label{table1:2}
\end{table}

We then consider the saddle dynamics for the stingray function \cite{Gra}
$$E(x_1,x_2)=x_1^2+(x_1-1)x_2^2 $$
and compute its index-1 saddle point with the initial conditions
$$x_0=
\left[\!\!
\begin{array}{c}
	\ds 0\\
	\ds 1
\end{array}
\!\!\right],~~
v_{1,0}=\frac{1}{\sqrt{5}}
\left[\!\!
\begin{array}{c}
	\ds 1\\
	\ds 2
\end{array}
\!\!\right]
$$
and index-2 saddle point with the initial conditions
$$x_0=
\left[\!\!
\begin{array}{c}
	\ds 0\\
	\ds 1
\end{array}
\!\!\right],~~
v_{1,0}=\frac{1}{\sqrt{5}}
\left[\!\!
\begin{array}{c}
	\ds 1\\
	\ds 2
\end{array}
\!\!\right],~~
v_{2,0}=\frac{1}{\sqrt{5}}
\left[\!\!
\begin{array}{c}
	\ds -2\\
	\ds 1
\end{array}
\!\!\right].
$$
The parameters are chosen as before and  numerical results are presented in Tables \ref{table2:1}-\ref{table2:2}, which again show the first-order accuracy of the semi-implicit scheme (\ref{FDsadk}) as proved in Theorem \ref{thmevk}.

\begin{table}[h]
	\setlength{\abovecaptionskip}{0pt}
	\centering
	\caption{Convergence rates of computing index-1 saddle point of stingray function.}
	\vspace{-0.25em}	
	\begin{tabular}{c|cc|cc} \cline{1-5}
		$\tau$& $\max_n \|e^x_n\|$ & CR &  $\max_n \|e^{v_1}_n\|$ &CR\\
		\cline{1-5}		
		1/32&	3.63E-02&		&8.12E-03	&\\
		1/64&	1.78E-02&	1.03&	4.14E-03&	0.97\\
		1/128&	8.78E-03&	1.02&	2.08E-03&	1.00\\
		1/256&	4.31E-03&	1.03&	1.03E-03&	1.01\\
		\hline
	\end{tabular}
	\label{table2:1}
\end{table}

\begin{table}[h]
	\setlength{\abovecaptionskip}{0pt}
	\centering
	\caption{Convergence rates of computing index-2 saddle point of stingray function.}
	\vspace{-0.25em}	
	\begin{tabular}{c|cc|cc|cc} \cline{1-7}
		$\tau$& $\max_n \|e^x_n\|$ & CR &  $\max_n \|e^{v_1}_n\|$ &CR& $\max_n \|e^{v_2}_n\|$ &CR\\ \cline{1-7}		
		1/32&	8.69E-02&		&   2.04E-02&	    &    2.04E-02	&\\
		1/64&	4.55E-02&	0.93&	1.04E-02&	0.98&	1.04E-02	&0.98\\
		1/128&	2.31E-02&	0.97&	5.19E-03&	1.00	 &   5.19E-03&	1.00\\
		1/256&	1.15E-02&	1.00	&   2.57E-03&	1.02	&2.57E-03&	1.02\\
		\hline
	\end{tabular}
	\label{table2:2}
\end{table}		

\subsection{Comparison between SI and EX schemes in finding saddle points}
In this experiment we compare the behavior of semi-implicit (denoted by ``SI'') and explicit (denoted by ``EX'') schemes (\ref{FDsadk}) and (\ref{schemeexp}) based on a Rosenbrock type function
$$E_R(x_1,x_2)=a(x_2-x_1^2)^2+b(1-x_1)^2. $$

\textit{Comparison 1: Pathway convergence.}
Let $(a,b)=(-30,0.5)$ and in this case $x_*:=(1,1)$ is an index-1 saddle point. The Hessian of $E_R$ at $x_*$ has two eigenvalues $-0.20$ and $299.20$, which leads to the Hessian condition number about $-1496$. We select the initial values as
\begin{equation}\label{testini}
	x_0=
	\left[\!\!
	\begin{array}{c}
		\ds 0.5\\
		\ds 0.5
	\end{array}
	\!\!\right],~~
	v_{1,0}=\frac{1}{\sqrt{2}}
	\left[\!\!
	\begin{array}{c}
		\ds 1\\
		\ds 1
	\end{array}
	\!\!\right]
\end{equation}
and compute three curves in Figure \ref{fig1}(left), which shows that though the numerical solutions computed under both the semi-implicit and explicit schemes could reach the saddle point $x_*$, the SI solution provides a much better approximation for the reference solution of the searching pathway (i.e. EX solution under $\tau=1/3000$ shown in Figure \ref{fig1}(left)) than the EX solution under the same step size $\tau=1/300$. Furthermore, the EX solution under $\tau=1/3000$ and the IS solution under $\tau=1/300$ have almost the same curves, which indicates that the semi-implicit scheme admits a much larger step size (10 times of that for explicit scheme in this example).

\begin{figure}[h!]
	\setlength{\abovecaptionskip}{0pt}
	\centering	\includegraphics[width=2.3in,height=2.1in]{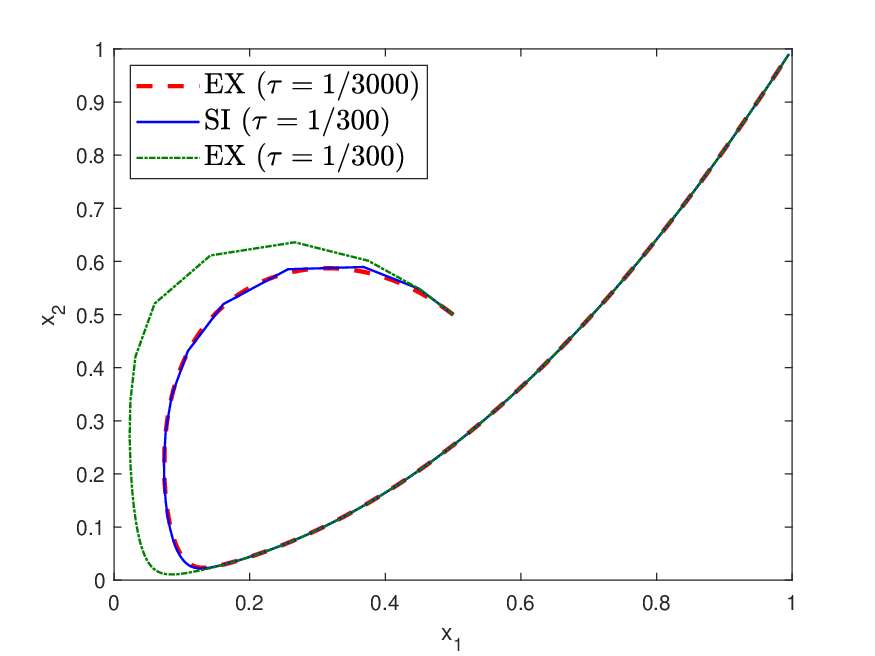}
	\includegraphics[width=2.3in,height=2.1in]{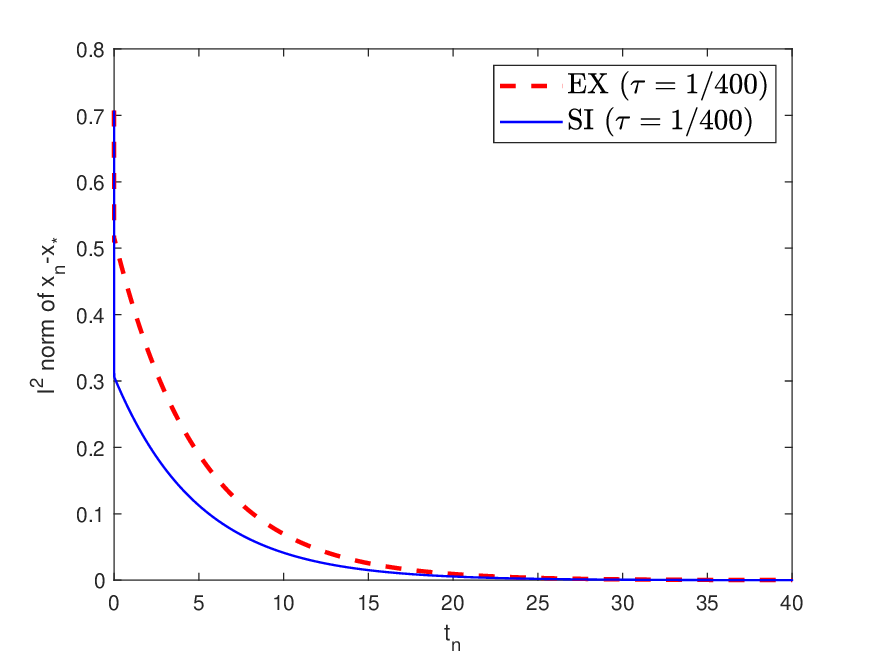}
	\caption{(left) Numerical trajectories of $x(t)=(x_1(t),x_2(t))$ under $T=20$ and different step sizes and schemes; (right) Plots of $\|x_n-x_*\|$ under $T=40$ and different schemes.}
	\label{fig1}
\end{figure}

\begin{figure}[h!]
	\setlength{\abovecaptionskip}{0pt}
	\centering	\includegraphics[width=2.3in,height=2.1in]{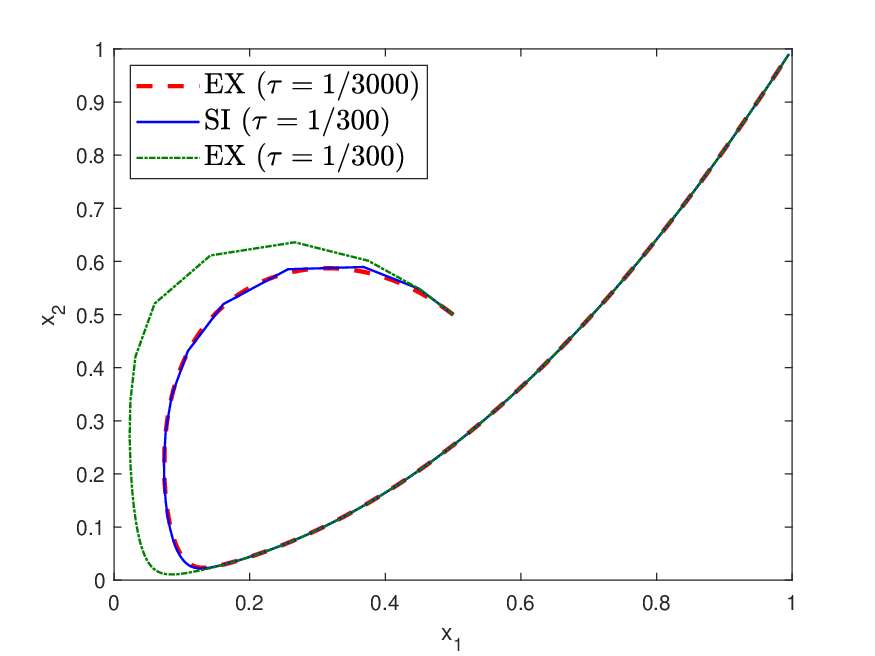}
	\includegraphics[width=2.3in,height=2.1in]{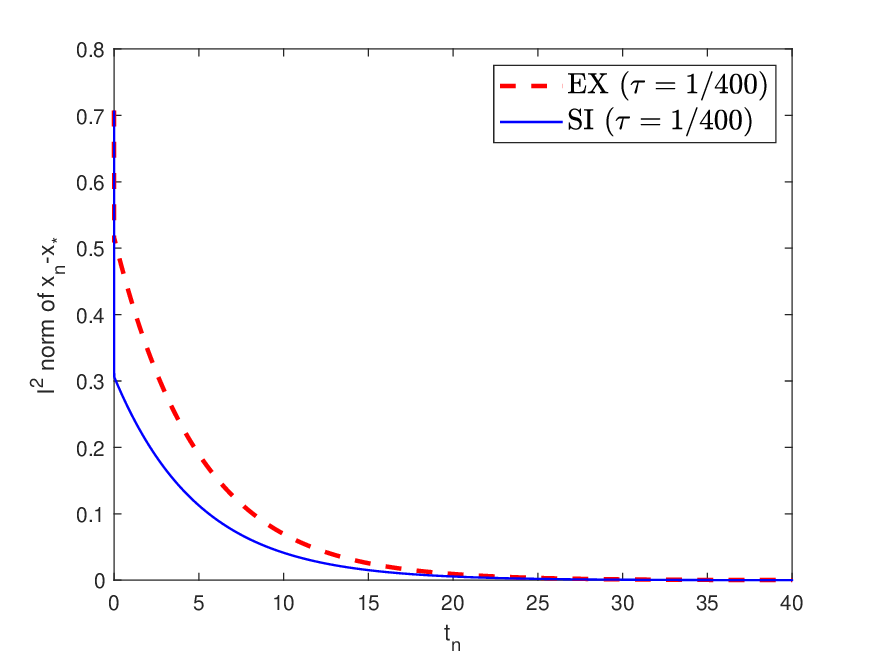}
	\caption{(left) Numerical trajectories of $x(t)=(x_1(t),x_2(t))$ under $T=20$ and different step sizes and schemes; (right) Plots of $\|x_n-x_*\|$ under $T=40$ and different schemes.}
	\label{fig1}
\end{figure}
\textit{Comparison 2: Convergence to saddle point.}
We set $(a,b)=(0.1,-300)$ in $E_R(x,y)$ and $x_*=(1,1)$ is again an index-1 saddle point. The Hessian of $E_R$ at $x_*$ has two eigenvalues $-0.20$ and $599.20$, which implies a large Hessian condition number about $-2992$. We choose the same initial values as (\ref{testini}) and present the distance $\|x_n-x_*\|$ computed under EX and SI in Figure \ref{fig1}(right), which indicates that under the same step size $\tau=1/400$, the SI solution approaches the saddle point $x_*$ much faster than the EX solution. All the observations in this section demonstrate the advantages of the proposed semi-implicit scheme compared with the commonly-used explicit scheme.

\subsection{Comparison between SI and EX schemes in constructing solution landscapes}
We compare the solution landscapes of the following Allen-Cahn equation computed by both methods
\begin{equation}\label{Allen Cahn}
	\dot{u} =F(u):=\kappa u_{xx}+u-u^3,
\end{equation}
where $-1\leq u\leq1$ is an order parameter and $\kappa$ is the diffusion coefficient. The computation region is $[0,1]$ with the uniform mesh size $2^{-7}$ in discretization of spatial operators. For any $\kappa>0$, $u_1\equiv1,u_{-1}\equiv-1$ and $u_0\equiv0$ are three stationary solutions of equation (\ref{Allen Cahn}). It is clear that both $u_1$ and $u_{-1}$ are steady states, while $u_0$ is the highest-index saddle point with different index for different $\kappa$. Using $u_1$ as the root state, we compute the solution landscape by the upward search \cite{YinSCM} with both semi-implicit schemes and explicit schemes, where the multiplications of the Hessian and the vector in schemes are approximated by the dimer method \cite{YinSISC}. Due to the high dimension of the problem, the GMRES method is applied for solving the SI scheme.

We first compare the largest step sizes $\tau$ in both methods that guarantee the convergence of the schemes under different $\kappa$ and indexes of latent saddle points (denoted by ``$k$-SD''). The numerical results are shown in the Table \ref{step}, which indicates that much larger time steps could be used in the SI scheme than those in the EX scheme, e.g. the former could be hundreds or even thousands of times more than the later. Furthermore, for a fixed $\kappa$, the maximum steps of the EX scheme keep nearly unchanged with the increment of the index, while those for the SI scheme become larger, which suggests that the SI scheme has greater advantage (more stable and efficient) when finding saddle points with higher index.

\begin{table}[h]
	\setlength{\abovecaptionskip}{0pt}
	\centering
	\caption{Maximal step sizes $\tau$ for the SI and EX schemes.}
	\begin{tabular}{c|c|c|c|c|c}
		\hline
		SI& $\kappa=0.02$ & $\kappa=0.005$ & $\kappa=0.0025$ & $\kappa=0.00125$ & $\kappa=0.001$ \\ \hline
		11-SD &               &                &                 &                  & 3.83           \\
		9-SD   &               &                &                 & 2.82             & 2.56           \\
		7-SD   &               &                & $>5$            & 2.62             & 1.86           \\
		5-SD   &               & $>$5           & 3.63            & 1.15             & 7.44E-01        \\
		3-SD   & 2.78          & 2.12           & 1.52            & 5.91E-01          & 5.61E-01        \\
		1-SD  & 1.03          & 1.00           & 6.97E-01        & 5.02E-01          & 5.02E-01        \\ \hline
	\end{tabular}
	
	\begin{tabular}{c|c|c|c|c|c}
		\hline
		EX      & $\kappa=0.02$ & $\kappa=0.005$ & $\kappa=0.0025$ & $\kappa=0.00125$ & $\kappa=0.001$ \\ \hline
		11-SD  &               &                &                 &                  & 1.24E-02        \\
		9-SD &               &                &                 & 9.94E-03          & 1.23E-02        \\
		7-SD   &               &                & 4.95E-03         & 9.91E-03          & 1.23E-02        \\
		5-SD   &               & 2.47E-03        & 4.95E-03         & 9.88E-03          & 1.23E-02        \\
		3-SD   & 6.14e-4       & 2.47E-03        & 4.93E-03         & 9.84E-03          & 1.23E-02        \\
		1-SD   & 6.14e-4       & 2.47E-03        & 4.93E-03         & 9.81E-03          & 1.23E-02        \\ \hline
	\end{tabular}
\label{step}
\end{table}

We then compare the solution landscapes with different $\kappa$ and numerical schemes in Figure \ref{land}. Each image in Figure \ref{land} represents a stationary solution. Because of the periodic boundary condition, a stationary solution after translation is still a stationary solution, and only one phase is presented. Each column with several images displays a solution landscape. We present three pairs of columns with different $\kappa$ in Figure \ref{land}, and the two columns in each pair are solution landscapes computed by different schemes. For $\kappa=0.02, 0.0025, 0.001$, the step sizes for EX scheme are 0.0005, 0.004, 0.01, respectively, while  we pick $\tau=0.4$ for SI scheme for all cases. It is clear from Figure \ref{land} that both methods generate almost the same solution landscapes, while the SI scheme admits much larger step size.
\begin{figure}[h!]
	\setlength{\abovecaptionskip}{0pt}
	\centering	\includegraphics[width=4.8in,height=4.8in]{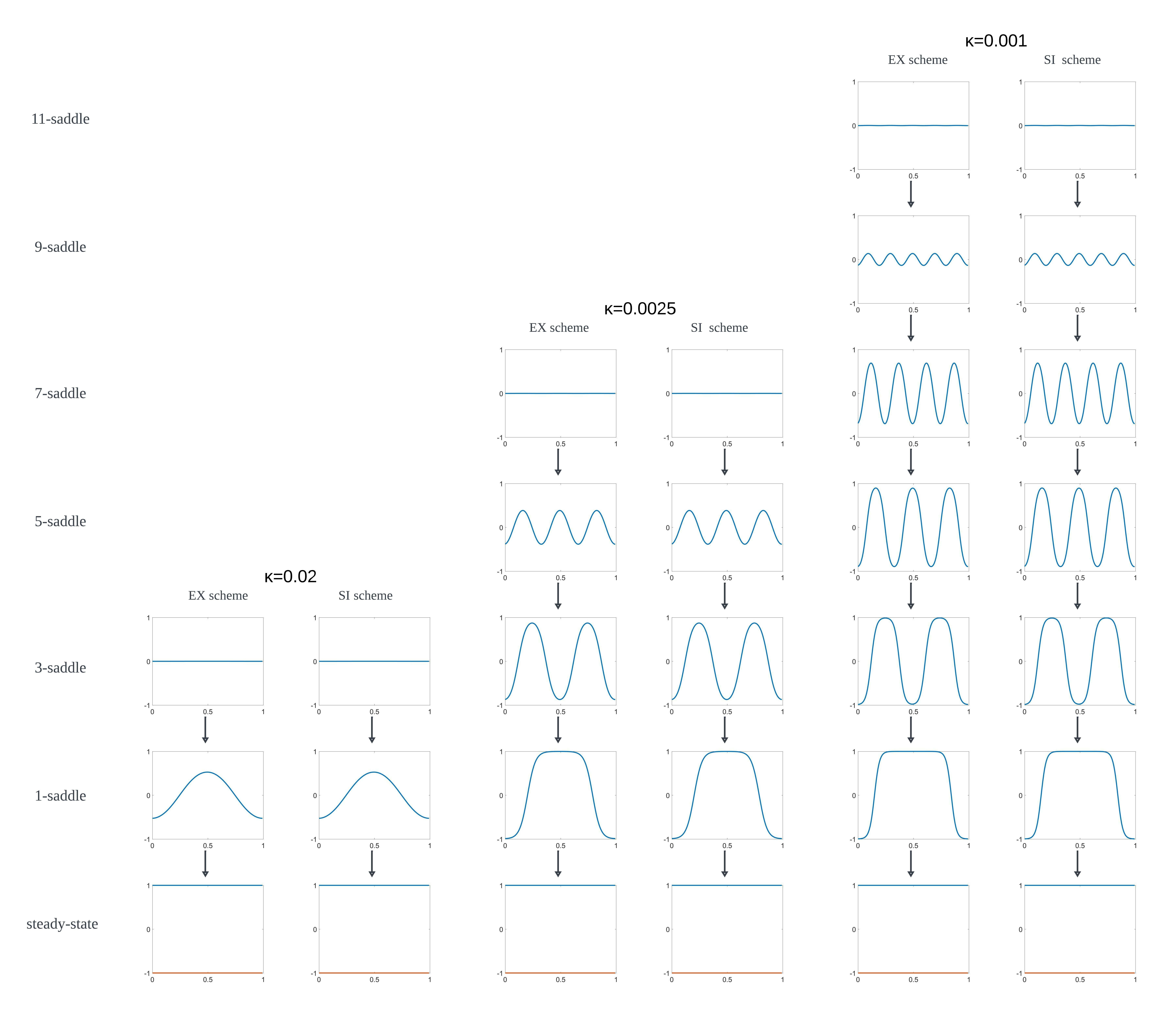}
	\caption{Comparison of solution landscapes computed by SI and EX schemes for the phase field model (\ref{Allen Cahn}) under different $\kappa$. For $\kappa=0.02, 0.0025, 0.001$, the step sizes for EX scheme are 0.0005, 0.004, 0.01, respectively, while $\tau=0.4$ for SI scheme for all cases. All these step sizes are chosen based on the maximal step sizes proposed in Table \ref{step}.}
	\label{land}
\end{figure}

To further demonstrate the advantages of the SI method, we compare the CPU times (denoted by ``CPU''), the numbers of iterations (denoted by ``$N_{iter}$'') and the number of queries of $F$ (denoted by ``$N_F$'') in the two schemes until the convergence of the algorithm (the stopping criteria is always chosen as $\|F\|\leq 10^{-4}$). Here the $N_{iter}$ for EX scheme implies the number of time steps until convergence. In practical problems, evaluating $F$ could be expensive or time-consuming and thus we expect to reduce $N_F$ by SI method. We fix $\kappa=0.001$ and the step sizes for SI and EX schemes are chosen as 0.4 and 0.012, respectively, where the later is nearly the maximal allowed step given in Table \ref{step}. Numerical experiments are presented in Table \ref{eval}, which indicate that the SI method significantly reduces the CPU times and numbers of iterations and queries of $F$, which fully demonstrate the advantages of the SI method.
\begin{table}[h]
	\setlength{\abovecaptionskip}{0pt}
	\centering
	\caption{Comparison of CPU, $N_{iter}$ and $N_F$ between the SI and EX schemes.}
	\begin{tabular}{c|ccc|ccc}
		\hline
		& \multicolumn{3}{c}{EX with $\tau=0.012$} & \multicolumn{3}{|c}{SI with $\tau=0.4$} \\
		\hline
		Index & $N_{iter}$    & $N_F$ &CPU(s)   & $N_{iter}$   & $N_F$ &CPU(s)  \\ \hline
		11  & {35148}   & 808404 &52.7  & {1047}   & 170684 &14.2 \\
		9   & {16147}   & 306793 &21.3  & {479}    & 75221  &5.94 \\
		7   & {6559}    & 98385  &5.76 & {234}    & 43707  &3.64 \\
		5   & {26764}   & 294404 &16.9  & {988}    & 123447 &10.9 \\
		3   & {66305}   & 464135 &26.0  & {2511}   & 100338 &10.8 \\
		1   & {81567}   & 244701 &14.3  & {3116}   & 94044  &9.97 \\ \hline
	\end{tabular}
	\label{eval}
\end{table}

\section{Concluding remarks}
In this paper we prove the error estimates for the semi-implicit scheme of high-index saddle dynamics and the convergence of the GMRES solver, which provides theoretical supports for numerical implementation of saddle dynamics in, e.g., constructing the solution landscape. The main contribution lies in developing novel analysis such as the multi-variable circulating induction procedure in Theorem \ref{lem2k} to accommodate the ``loss of orthonormal property'' on the schemes of $\tilde v_{i,n}$ in (\ref{FDsadk}) and the coupling between the computation of eigenvectors and the orthonormalization procedure. Extensive numerical experiments are carried out from different aspects to compare the explicit and semi-implicit schemes and demonstrate the advantages of the later.

It is worth mentioning that the developed methods could be naturally extended to analyze the semi-implicit numerical scheme of generalized high-index saddle dynamics for dynamic (non-gradient) systems \cite{GuZho,YinSCM}, i.e.,
\begin{equation}\label{ngsadk}
	\left\{
	\begin{array}{l}
		\ds \frac{dx}{dt} =\bigg(I -2\sum_{j=1}^k v_jv_j^\top \bigg)F(x),\\[0.075in]
		\ds \frac{dv_i}{dt}=(I-v_iv_i^\top)\mathcal J(x)v_i -\sum_{j=1}^{i-1}v_jv_j^\top(\mathcal J(x)+\mathcal J(x)^\top)v_i,~~1\leq i\leq k.
	\end{array}
	\right.
\end{equation}
Here $\mathcal J(x)$ refers to the Jacobian of $F(x)$, which is in general not symmetric. Compared with the high-index saddle dynamics (\ref{sadk}) for the gradient systems with the symmetric Hessian $J(x)$, $2J(x)$ in the dynamics of $\{v_i\}_{i=1}^k$ in (\ref{sadk}) is substituted by the symmetrization $\mathcal J(x)+\mathcal J(x)^\top $, which is the key to ensure the validity of the preceding derivations, especially the Lemma \ref{lemm1}. To be specific, for the generalized high-index saddle dynamics (\ref{ngsadk}), the last right-hand side term of (\ref{conc1}) will become
$$-\tau\gamma v_{m,n}^\top (\mathcal J(x_n)+\mathcal J(x_n)^\top) \tilde v_{i,n}, $$
where the $\mathcal J(x_n)$ and $\mathcal J(x_n)^\top$ exactly match those in $A_1$ and $A_3$, respectively. Consequently the $B_1$ and $B_2$ in (\ref{conc2}) become
$$\begin{array}{l}
	\ds B_1+B_2=\tau\gamma\big(v_{m,n-1}^\top \mathcal J(x_n)\tilde v_{i,n}-v_{m,n}^\top \mathcal J(x_n) \tilde v_{i,n} \big)\\[0.1in]
	\ds\qquad\qquad\qquad+\tau\gamma\big(\tilde v_{m,n}^\top \mathcal J(x_n)^\top v_{i,n-1}-v_{m,n}^\top \mathcal J(x_n)^\top \tilde v_{i,n}\big),
\end{array}  $$
which essentially avoids the differences like $\mathcal J(x_n)^\top - \mathcal J(x_n)$ that do not generate the desired numerical accuracy. In summary, by virtue of the symmetrization, the numerical accuracy is preserved throughout the proof.

There are several other potential extensions of the current work that deserve further exploration. For instance, the techniques could be employed and improved to analyze the semi-implicit numerical scheme for constrained high-index saddle dynamics \cite{CHiSD2021,ZhaDu2}:
\begin{equation}
	\label{ccsd}
	\left\{
	\begin{array}{l}
		\ds \frac{dx}{dt} =\bigg(I -2\sum_{j=1}^k v_jv_j^\top \bigg)F(x),\\[0.075in]
		\ds \frac{dv_i}{dt}=\bigg(I-v_iv_i^\top-2\sum_{j=1}^{i-1}v_jv_j^\top\bigg)\mathcal H(x)[v_i] \\
		\ds\qquad\qquad-A(x)\big(A(x)^\top A(x)\big)^{-1}\bigg(\nabla^2c(x)\frac{dx}{dt}\bigg)^\top v_i,~~1\leq i\leq k.
	\end{array}
	\right.
\end{equation}
Here $c(x)=(c_1(x),\cdots,c_m(x))=0$ represents the $m$ equality constraints and $A(x)=(\nabla c_1(x),\cdots,\nabla c_m(x))$.
In the constrained high-index saddle dynamics (\ref{ccsd}), $\mathcal H(x)$ refers to the Riemannian Hessian, which is difficult to analyze and approximate in practice and thus brings additional difficulties for the numerical analysis that we will investigate in the near future.
\bibliographystyle{amsplain}

\end{document}